\documentclass[a4paper]{amsart}

\usepackage{amsfonts, latexsym, amssymb, amsthm, amsmath, mathrsfs, esint}

\newcommand{\LA}{\left<}
\newcommand{\RA}{\right>}
\newcommand{\LB}{\left[}
\newcommand{\RB}{\right]}
\newcommand{\LV}{\left|}
\newcommand{\RV}{\right|}
\newcommand{\LC}{\left(}
\newcommand{\RC}{\right)}
\newcommand{\m}{\boldsymbol{m}}
\newcommand{\p}{\partial}
\newcommand{\f}{\boldsymbol{f}}
\newcommand{\Stwo}{\mathbb{S}^2}
\newcommand{\e}{\epsilon}
\newcommand{\logep}{\log {1\over \e}}
\newcommand{\R}{\mathbb{R}}
\newcommand{\N}{\mathbb{N}}
\newcommand{\Z}{\mathbb{Z}}
\newcommand{\C}{\mathbb{C}}

\DeclareMathOperator*{\osc}{osc}

\newcommand{\BMO}{\mathrm{BMO}}

\DeclareMathOperator{\supp}{supp}
\newcommand{\loc}{\mathrm{loc}}
\newcommand{\scp}[2]{\left\langle #1, #2 \right\rangle}
\newcommand{\dd}[2]{\frac{\partial #1}{\partial #2}}
\newcommand{\set}[2]{\left\{ #1 \, \colon \ #2 \right\}}

\renewcommand{\div}{\mathop{\mathrm{div}}\nolimits}
\newcommand{\curl}{\mathop{\mathrm{curl}}\nolimits}

\newcommand{\restr}{\mathchoice
{\kern2pt\mbox{\vrule width 0.08ex height1.5ex depth0ex\kern-0.08ex\vrule width 1.5ex height.08ex depth0ex}\kern2pt}
{\kern2pt\mbox{\vrule width 0.08ex height1.5ex depth0ex\kern-0.08ex\vrule width 1.5ex height.08ex depth0ex}\kern2pt}
{\kern1.5pt\mbox{\vrule width 0.06ex height1.1ex depth0ex\kern-0.06ex\vrule width 1.1ex height.06ex depth0ex}\kern1.5pt}
{\kern1pt\mbox{\vrule width 0.04ex height0.75ex depth0ex\kern-0.04ex\vrule width 0.75ex height.04ex depth0ex}\kern1pt}
}

\newtheorem{theorem}{Theorem}[section]
\newtheorem{lemma}{Lemma}[section]

\theoremstyle{remark}
{}
\newtheorem*{remark}{Remark}

\begin{document}
\title[Excess energy and LLG vortex dynamics]{Vortex dynamics in the presence of  excess energy for the Landau-Lifshitz-Gilbert equation}

\thanks{ M. Kurzke was supported by DFG SFB 611.  D.~Spirn was partially supported by NSF grant DMS-0955687.}

\author[M. Kurzke]{Matthias Kurzke}
\address{Institute for Applied Mathematics \\ University of Bonn \\ Endenicher Allee 60 \\ 53115 Bonn \\ Germany}
\email{kurzke@iam.uni-bonn.de}

\author[C. Melcher]{Christof Melcher}
\address{Department of Mathematics I \\ RWTH Aachen University \\52056 Aachen \\ Germany}
\email{melcher@rwth-aachen.de}

\author[R. Moser]{Roger Moser}
\address{Department of Mathematical Sciences \\ University of Bath \\ Bath BA2 7AY \\ United Kingdom}
\email{r.moser@bath.ac.uk}

\author[D. Spirn]{Daniel Spirn}
\address{School of Mathematics \\ University of Minnesota \\ Minneapolis, MN 55455 \\ USA}
\email{spirn@math.umn.edu}

\date{\today}

\begin{abstract}
We study the Landau-Lifshitz-Gilbert equation for the dynamics of a magnetic vortex system.  
We present a PDE-based method for proving vortex dynamics  that 
does not rely on strong well-preparedness of the initial data and allows for instantaneous changes in the strength of the gyrovector force due to bubbling events.  The main tools are  estimates of the Hodge decomposition of the supercurrent and 
an analysis of the defect measure of weak convergence of the
 stress energy tensor. Ginzburg-Landau equations with mixed dynamics in the presence of excess energy are also discussed.

\end{abstract}

\maketitle

\section{Introduction}
Ferromagnets in a domain $\Omega \subset \R^3$ are typically modeled by a magnetization vector
$\m: \Omega \to \mathbb{S}^2$ with values on the unit 2-sphere.   Dynamically, the magnetization satisfies the 
Landau-Lifshitz-Gilbert (LLG) equation:
\begin{equation}  \label{preLLG}
{\p \m \over \p t} = \m \times \LC  \alpha {\p \m \over \p t} - \boldsymbol{h}_{\mathrm{eff}} \RC,
\end{equation}
where $\boldsymbol{h}_{\mathrm{eff}} $ is the effective  field, arising from the $L^2$ gradient of the 
micromagnetic energy $E_\e(\m)$, and $\alpha > 0$ is a damping parameter.  The  form of $E_\e(\m)$ depends heavily on the physics of the ferromagnetic sample and contains
a nonlocal pseudodifferential operator; however, for thin, isotropic materials the energy  simplifies.  

When the domain is thin, the magnetization vector lies mostly in the plane; consequently,   
$\m = (m_1, m_2, m_3) =: (m, m_3)$ behaves roughly  like $\m \approx (m, 0)$ where $m \in \mathbb{S}^1$.  
Furthermore,  constraining $\m$ onto the plane induces the formation of \emph{vortices}, points  $(a_1,\ldots,a_N) \in \Omega^N$,  about which the winding number of the planar component, $m$, is quantized.  In the center of each 
vortex, the micromagnetic vector $\m \approx (0,0, \pm1)$; and hence, such vortices carry 
both the $\mathbb{S}^1$-degree of the winding number and a \emph{polarity}, which may be interpreted as an $\mathbb{S}^2$-degree and is in simple
situations given by the value of $m_3$ at the center of the vortex.  Hence, micromagnetic vortices carry two pieces of information, whereas vortices arising 
in  superconductivity, superfluids, and Bose-Einstein condensates have only the winding number. 

The  micromagnetic energy and the local area element  concentrate at the site of the vortices $(a_1,\ldots,a_N)$.  In particular the local area element, described by the \emph{micromagnetic vorticity}, concentrates via a cover of a hemisphere at the site of a vortex.  The question of how   such concentrated quantities behave in \eqref{preLLG} has been a rich field of study.   Thiele \cite{Thiele73} and Huber \cite{Huber82} showed formally that in the dynamical setting
concentrations obey an equation of the form
\[
F_n + G_n \times \dot{a}_n - \alpha_0 \dot{a}_n  =0
\]
with an \emph{interaction force} $F_n = F_n(a_1, \ldots, a_N)$, a \emph{gyrovector force} $G_n = 4 q_n \mathbf{e}_3$, and a
\emph{damping term} $\alpha_0 \dot{a}_n$.

In a previous paper \cite{Kurzke-Melcher-Moser-Spirn:11}, the authors gave further justification
of this motion law by deriving it rigorously from the LLG equation, but only under strong assumptions
on the initial data and for boundary conditions, chosen mainly for technical rather than physical reasons. 
A similar problem with an additional spin transfer torque was studied in a different paper \cite{Kurzke-Melcher-Moser:11} by the first three authors.   The undamped problem was studied by Lin-Shatah \cite{Lin-Shatah:03} and later by Lin-Wei \cite{Lin-Wei:10}, who found traveling wave solutions for vortex pairs.  

Here we present a new approach that works under significantly weaker assumptions and for
other types of boundary conditions. While the previous work relied heavily on variational
methods and ideas of Lin-Xin \cite{Lin-Xin:99}, Colliander-Jerrard \cite{Colliander-Jerrard:99}, and Sandier-Serfaty \cite{Sandier-Serfaty:04.1,Sandier-Serfaty:04.2},
we now mostly use the structure of the PDE coming from \eqref{preLLG} and compensated
compactness arguments in the spirit of H\'elein \cite{Helein:02} and Lin-Rivi\`ere \cite{Lin-Riviere:01}.

\subsection{Mathematical setting and results} \label{results}

We use essentially the same terminology and notation as in our previous paper
\cite{Kurzke-Melcher-Moser-Spirn:11}.  Therefore, we keep their discussion brief.

We use the functional
\begin{equation}\label{energy}
E_\e(\m) = \int_\Omega e_\e(\m) \, dx
\end{equation}
with
\begin{equation}
e_\e(\m) = {1\over2} \LV \nabla \m \RV^2 + {1\over 2\e^2} m_3^2
\end{equation}
as a model for the micromagnetic energy. Formally when $\e$ is very small and the energy is not too large, then
$\m$ must take values close to the equator $\mathbb{S}^1 \times \{0\}$ in most of $\Omega$. For topological
reasons it may be forced to reach the poles at certain points, but the third component will
decay rapidly away from these points. In particular one expects a sequence of local minimizers of \eqref{energy} to converge 
to a map $\m_* = (m_*,0)$ away from a finite number of points $a_1,\ldots,a_N$ in $\Omega$.   

In the vicinity of each vortex the energy blows up at a rate $\pi \logep$ up to lower order terms.  Hang-Lin \cite{Hang-Lin:01},
making use of arguments of Bethuel-Brezis-H\'elein \cite{Bethuel-Brezis-Helein:94} for the associated Ginzburg-Landau energy, showed that 
as $\e\to 0$,
\[
E_\e(\m) = N \LC \pi \logep + \gamma \RC + W(a) + o(1),
\]
where $\gamma$ is a universal constant and 
\[
W(a) = - \pi \sum_{n \neq m} \log \LV a_n - a_m \RV + \hbox{boundary effects}
\]
is a \emph{renormalized energy} when the winding number about each vortex is one.  The renormalized energy is the Kirchoff-Onsager functional arising in Euler point vortices.

We allow winding numbers $\pm 1$ in this paper; therefore, we  consider pairs
$(a_n,d_n)$ of points $a_n \in \Omega$ and winding numbers $d_n = \pm 1$
for $n = 1,\ldots,N$. Writing $a = (a_1,\ldots,a_n)$ and $d = (d_1,\ldots,d_n)$, we
obtain a renormalized energy $W(a,d)$. This will also depend on the boundary
conditions, and we study Dirichlet as well as Neumann boundary data. As the
renormalized energy is the same as in the theory of Ginzburg-Landau vortices
and is discussed extensively elsewhere \cite{Bethuel-Brezis-Helein:94}, we do not write
down the details here.

Due to the blowup of energy at the site of each vortex, one expects that 
\[
{e_\e(\m_\e) \over \logep} \to \pi \sum_{n = 1}^N \delta_{a_n}
\]
in distribution when we have a family of maps with good properties. 
Apart from the energy density, another fundamental quantity associated to the micromagnetic vector $\m$ is the 
\emph{vorticity},
\[
\omega(\m) = \LA \m, {\p \m \over \p x_1} \times { \p \m \over \p x_2} \RA,
\]
which can be viewed as the Jacobian of the mapping $\m: \Omega \to \Stwo$; consequently, it measures the local signed area of the mapping.  If $\m$ is smooth on a disk $B_r(x)$ and $\left. \m  \right|_{\p B_r(x)}$ takes values on the equator $\mathbb{S}^1 \times \{0\}$ with
winding number $d$, then
\[
\int_{B_r(x)} \omega(\m) \, dx \in 2\pi d + 4 \pi \mathbb{Z}.
\]
Due to the ``relaxed constraint'' which induces $\m$ to lie mostly in $\mathbb{S}^1$, one finds that the vorticity will concentrate at a number of points,
each point giving rise to a concentration of the vorticity of the amount $2\pi d + 4\pi q$; namely the area of $d$ hemispheres plus the area of a number of ``bubbles'', or covers of $\mathbb{S}^2$.  Thus we expect
\[
\omega(\m_\e) \to 2\pi \sum_{n = 1}^N d_n \delta_{a_n} + 4\pi \sum_{n=1}^N \hat{q}_n \delta_{a_n} + 4 \pi \sum_{p=1}^P \tilde{q}_p \delta_{b_p},
\]
where $\hat{q}_n \in \mathbb{Z}$ correspond to bubbles at the vortices and $\tilde{q}_p \in \mathbb{Z}$ to bubbles elsewhere.
Under strict conditions on the amount of energy, one can show that  $\hat{q}_n = 0$ and $\tilde{q}_p =0$ for all $n$ and $p$, see \cite{Kurzke-Melcher-Moser-Spirn:11}.

In order to separate the winding numbers $d_n$ from the $\mathbb{S}^2$ degrees $q_n$, we also consider
the planar Jacobian
\[
J(m) = \dd{m_1}{x_1} \dd{m_2}{x_2} - \dd{m_1}{x_2} \dd{m_2}{x_1}, 
\]
and we expect that
\[
J(m_\e) \to \pi \sum_{n=1}^N d_n \delta_{a_n}
\]
as well. Studying the Jacobian also helps for the analysis since we can apply standard results
from the theory of Ginzburg-Landau vortices to $J(m)$.

The question of  how concentrations, described above, are moved by the  Landau-Lifshitz-Gilbert equation was considered by the authors in \cite{Kurzke-Melcher-Moser-Spirn:11}.
Using the model $E_\epsilon$ for the free energy and the abbreviation
\begin{equation} \label{L2-gradient}
\f_\epsilon(\m_\epsilon) = \Delta \m + \LV \nabla \m \RV^2 \m - {1\over \e^2} \LC m_3 \mathbf{e}_3 - m_3^2 \m \RC
\end{equation}
for the negative $L^2$ gradient, the equation takes the form
\begin{equation} \label{LLG}
\dd{\m_\epsilon}{t} =  \m_\epsilon \times \LC \alpha_\e \dd{\m_\epsilon}{t}  - \f_\epsilon(\m_\epsilon)  \RC,
\end{equation}
and the behavior of $\alpha_\epsilon$ for $\epsilon \to 0$ is crucial for the answer.
If it decays very slowly or not at all, then we expect no gyrovector force in the limit; on the other hand, if it decays
rapidly, then the gyrovector force will dominate and the damping term $\alpha_0 \dot{a}_n$ will
be invisible. The most interesting case is when both terms coexist. This is expected when $\alpha_\epsilon$
is of the order $1/|\log \epsilon|$. We assume that
\[
\alpha_\epsilon \log \frac{1}{\epsilon} \to \alpha_0 \quad \text{as } \epsilon \to 0.
\]
In the aforementioned paper \cite{Kurzke-Melcher-Moser-Spirn:11}, a combination of differential
identities for the time-evolution of $\omega(\m_\e)$ and $e_\e(\m_\e)$, 
along with special choices of test functions, yielded an ODE for the evolution of the vortex positions:
\begin{equation} \label{idealODE}
4 \pi q_n i \dot{a}_n + \pi \alpha_0 \dot{a}_n + {\p W \over \p a_n}(a,d) = 0,
\end{equation}
where $q_n = \pm \frac{1}{2}$. The proof of \eqref{LLG} in \cite{Kurzke-Melcher-Moser-Spirn:11} relies heavily on
the so-called well-preparedness of the initial data. That is, we define the \emph{excess energy} to be the difference
between the actual energy and an expression that describes asymptotically the energy needed to
develop the observed vortices. It is then assumed that the excess energy tends to $0$ as $\epsilon \to 0$.
Furthermore, the arguments require Dirichlet boundary data and
the assumption that all winding numbers $d_n = 2q_n$ are of the same sign. 


It is natural to ask what happens to concentrations in \eqref{LLG} when the initial data are not  well-prepared, and there are two important problems that occur.
The first issue is that we cannot prevent bubbling from occurring near vortices.  In \cite{Kurzke-Melcher-Moser-Spirn:11} the authors prevented this by strict energy control, which no longer holds.    Consequently,  the gyrovector value 
can spontaneously change.  The more troubling problem concerns the proof of the convergence
of the stress energy tensor to the stress energy tensor of the limiting canonical harmonic map $\m_*$ since energy bounds only imply weak convergence in $H^1$ away from the vortices.   To overcome the lack of strong convergence, the authors in \cite{Kurzke-Melcher-Moser-Spirn:11} prove bounds for $\int_{\Omega \backslash \cup B_{r}(a_n)} |\nabla \m_\e - \nabla \m_*|^2 \, dx$ in
terms of the excess energy; hence, if the initial data are well-prepared (i.e., the excess energy vanishes at the initial time), then $\m_\e$ converges strongly to $\m_*$ outside of the vortex cores.  Again this argument fails when the initial data are not well-prepared.

This paper develops  techniques to establish the vortex motion law of micromagnetic vortices in the presence of excess energy.
The arguments work when winding numbers of both signs are present and for both Dirichlet and Neumann boundary conditions.

Let $\boldsymbol{g} = (g,0) : \partial \Omega \to \mathbb{S}^1 \times \{0\}$ be a smooth map and consider smooth
initial data $\m_\epsilon^0 : \Omega \to \mathbb{S}^2$ with $\m_\epsilon^0|_{\partial \Omega} = \boldsymbol{g}$.
We assume that there exists a constant $C_0$ such that
\[
E_\epsilon(\m_\epsilon^0) \le \pi N \log \frac{1}{\epsilon} + C_0
\]
for every $\epsilon \in (0,\frac{1}{2}]$. Furthermore, we assume that there exist $a^0 = (a_1^0,\ldots,a_N^0) \in \Omega^N$
with $a_m \not= a_n$ for $m \not= n$, $d = (d_1,\ldots,d_N)$ with $d_n = \pm 1$, and $q^0 = (q_1^0,\ldots,q_N^0)$ with
$q_n^0 \in \frac{1}{2} + \Z$, such that
\[
\frac{e_\epsilon(\m_\epsilon^0)}{\log \frac{1}{\epsilon}} \to \pi \sum_{n = 1}^N \delta_{a_n^0}, \quad
J(m_\epsilon^0) \to \pi \sum_{n = 1}^N d_n \delta_{a_n^0}, \quad \text{and} \quad \omega(\m_\epsilon^0) \to 4\pi \sum_{n = 1}^N q_n^0 \delta_{a_n^0}
\]
in distribution.

These are the assumptions we use for the problem with Dirichlet boundary conditions. When
we study Neumann data, then we drop the condition $\m_\epsilon^0|_{\partial \Omega} = \boldsymbol{g}$,
and all of the other assumptions remain the same.

\begin{theorem} \label{vortexmotionlawLLG}
For every $\epsilon \in (0,\frac{1}{2}]$, there exists a weak solution $\m_\epsilon$
of \eqref{LLG}, with $\m_\e(0) = \m_0^\e$ and $\m(t)|_{\partial \Omega} = \boldsymbol{g}$ for all $t > 0$,
that is smooth up to finitely many points in space-time. Furthermore, there exist a number
$T > 0$, a piecewise constant function $q : (0,T) \to (\frac{1}{2} + \Z)^N$ with finitely
many jumps, and a sequence $\epsilon_k \to 0$ such that for almost every $t \in (0,T)$,
\[
\frac{e_{\epsilon_k}(\m_\epsilon(t))}{\log \frac{1}{\epsilon_k}} \to \pi \sum_{n = 1}^N \delta_{a_n(t)}, \quad
J(m_{\epsilon_k}(t)) \to \pi \sum_{n = 1}^N d_n \delta_{a_n(t)},
\]
and
\[
\omega(\m_{\epsilon_k}(t)) \to 2\pi \sum_{n = 1}^N d_n \delta_{a_n(t)} + 4\pi \sum_{n = 1}^N q_n(t) \delta_{a_n(t)},
\]
where $a = (a_1,\ldots,a_N)$ solves
\[
4\pi q_n i \dot{a}_n + \pi \alpha_0 \dot{a}_n + \dd{W}{a_n}(a,d) = 0, \quad n = 1,\ldots,N,
\]
in $(0,T)$ with $a(0) = a^0$.

The same statement holds for homogeneous Neumann instead of Dirichlet boundary conditions.
\end{theorem}

One unusual feature of the vortex motion law is the possibility of spontaneous changes in the gyrovector, represented by the jumps of $q_n$.  For a single vortex in a disk this  would correspond to a sudden change in the direction or the speed of circulation of a vortex about the origin.

A consequence is that using only Theorem \ref{vortexmotionlawLLG}, we cannot predict the
trajectories from the initial data alone. This is also the reason why we make a statement
only for a sequence $\epsilon_k \to 0$. Unless it is possible to extract more information,
it is conceivable that for a different choice of $\epsilon_k$, the trajectories of the
vortices are different, and in such a situation we would not have convergence of the
entire family $\m_\epsilon$.

From a physical point of view, changes in the direction of 
rotation are not something completely unexpected. They 
are known to occur in field- or current-driven switching processes,
where a vortex is driven from its equilibrium position by an applied
field or current (spiralling outward), then changes its polarization
and spirals back into equilibrium after the external field or current
has been switched off, see e.g. \cite{Gliga08,
GuslienkLeeKim:2008a} for some related simulations. 

In order to prove the theorem, we need to establish two things.  First we show that the stress energy tensor $\nabla \m_\e \otimes \nabla \m_\e$ converges to $\nabla \m_* \otimes \nabla \m_*$ outside of a  defect measure concentrated at a set of delta functions.  Second we analyze how the defect measure interacts with the micromagnetic vortices by applying different test functions on the differential identities for $\omega(\m_\e)$ and $e_\e(\m_\e)$, and it is shown that the defect measure has no effect on the vortex dynamics.



\subsection{Ginzburg-Landau equation with mixed dynamics}

The Landau-Lifshitz-Gilbert equation given above describes a mostly-planar $\Stwo$ micromagnetic vector.  There is an analogous problem in complex Ginzburg-Landau theory in which an order parameter $u: \R^2 \to \C$ evolves according to an equation with mixed dynamics: 
\begin{equation} \label{mixedgl}
(\alpha_\e + i) \dd{u}{t} = \Delta u + {1\over \e^2} u \LC 1 - |u|^2 \RC,
\end{equation}
which is also a hybrid of gradient flow and Schr\"odinger dynamics.  
Smooth solutions to \eqref{mixedgl} satisfy an
energy dissipation equality 
\[
E_{\mathrm{gl}}(u(t)) + \alpha_\e  \int_0^t \int_\Omega \LV \dd{u}{t} \RV^2 dx ds 
= E_{\mathrm{gl}}(u(0)),
\]
where $E_{\mathrm{gl}}(u) = \int_\Omega e_{\mathrm{gl}}(u) \, dx$ and $e_{\mathrm{gl}}(u) = {1\over2} \LV \nabla u \RV^2 + {1\over 4\e^2} \LC 1 - |u|^2 \RC^2$.  
In the Ginzburg-Landau setting, a vortex is defined by the concentration of energy ${e_{g}(u) \over \logep }\to \sum \pi \delta_{a_n}$
and the concentration of the Jacobian $J(u) \to \sum \pi d_n \delta_{a_n}$, both of which are well understood, see \cite{Jerrard-Soner:02}.

The formal vortex motion law of E \cite{E94} was established rigorously by Miot \cite{Miot:08} in the plane and the authors of this paper \cite{Kurzke-Melcher-Moser-Spirn:08} on bounded domains.  In both papers the authors again use the strong well-preparedness of the initial data to 
show that the stress-energy tensor converges to the stress-energy tensor of the limiting canonical harmonic map  as $\e \to 0$.   
This trick, pioneered by Colliander-Jerrard \cite{Colliander-Jerrard:99} and Lin-Xin \cite{Lin-Xin:99} in the context of Ginzburg-Landau vortex dynamics, is reliant on excess energy estimates.  One drawback of this method is that the initial data 
must lie close to the optimal mapping in an $H^1$ sense.  This certainly fails to take advantage of any energy dissipation that 
is present in \eqref{mixedgl}, as is well understood in the purely dissipative problem.

When the dynamics are purely dissipative, much stronger results can be shown that remove the well-preparedness assumption on the initial data.  The first proofs of the vortex motion law for the Ginzburg-Landau heat equation $\alpha_\e \dd{u_\e}{t} = \Delta u_\e + {1\over \e^2} u_\e \LC 1 - |u_\e|^2 \RC$ 
by  Lin \cite{Lin:96} and Jerrard-Soner \cite{Jerrard-Soner:98} proved strong convergence of $\nabla u_\e \to \nabla u_*$, the limiting canonical harmonic map, away from the vortex cores by parabolic estimates, even when the excess energy is of order $O(1)$.  The proofs depend on using estimates of the form $\Delta u_\e + {1\over \e^2} u_\e \LC 1 - |u_\e|^2 \RC = o_\e(1)$, which arise from energy bounds.  Later refinements of the vortex motion law by Bethuel-Orlandi-Smets \cite{Bethuel-Orlandi-Smets:05} and Serfaty \cite{Serfaty:07} again used the purely dissipative nature of the equation.   A naive adaptation of these
 methods fail for \eqref{mixedgl}: while we have 
the same $L^2$ control of $\frac{\p u_\e}{\p t}$ as in the 
purely dissipative case, it is not possible to deduce smallness of 
$\Delta u_\e + {1\over \e^2} u_\e \LC 1 - |u_\e|^2 \RC$. The reason 
for the difference 
is that 
$|\alpha_\e+i|\to 1$ while $|\alpha_\e|\to 0$ as $\e\to 0$.

For initial data without vortices and close to a constant map, 
Miot \cite{Miot:10} studied \eqref{mixedgl} in the damped 
wave regime. It is unclear how this approach can be generalized
to initial data with vortices. 

Using similar arguments to the proof of Theorem~\ref{vortexmotionlawLLG}, we can establish the vortex motion law for \eqref{mixedgl} for initial data that are not well-prepared.

\begin{theorem} \label{vortexmotionlawmixedgl}
Let $u_\e(t)$ be a sequence of solutions to \eqref{mixedgl} with initial data $u_\e(0) = u_\e^0$ and either Dirichlet or homogeneous Neumann boundary data. Assume  that
$E_\e(u_\epsilon^0) \le N \pi \logep + C_0$ for some constant $C_0$ and  
\[
{e_\e(u_\e^0) \over  \logep} \to \sum_{n=1}^N \pi \delta_{a^0_n} \quad \hbox{and} \quad {J(u_\e^0) } \to \sum_{n=1}^N \pi d_n \delta_{a^0_n}
\]
with $d_n = \pm 1$ for $n = 1,\ldots,N$.
Then there exists $T > 0$ such that for all $t \in (0, T)$
\begin{equation}
{e_\e(u_\e(t)) \over  \logep} \to \sum_{n=1}^N \pi \delta_{a_n(t)} \quad \hbox{and} \quad {J(u_\e(t)) } \to \sum_{n=1}^N \pi d_n \delta_{a_n(t)},
\end{equation}
where $a_n(t)$ solves
\begin{equation}  \label{mixedglODE}
2 \pi d_n i  \dot{a}_n + \pi \alpha_0 \dot{a}_n + {\p W \over \p a_n}(a,d) = 0
\end{equation}
for $n = 1,\ldots,N$.  
\end{theorem}

Therefore, the vortex motion law is exactly as in \cite{Kurzke-Melcher-Moser-Spirn:08}, and, unlike in the micromagnetic case, there are no changes in the direction of rotation of a single vortex about the origin in disk domains.  Since we are unable to handle vortex collisions or a vortex migrating to the boundary,  the $T$ in the theorem represents the first time for which either of these two events occur.  Finally, we note that the limiting motion law \eqref{mixedglODE} is independent of subsequence and uniquely determined by the initial data, unlike in the LLG case.  The proof of Theorem~\ref{vortexmotionlawmixedgl} is very similar to the   proof of Theorem~\ref{vortexmotionlawLLG}, and so we only briefly discuss the proof in section~\ref{mixedglsection}. 

\section{Mathematical tools}

Here we provide some mathematical notation, discuss  topological quantities arising in our study, and review facts about the expansion of the micromagnetic energy.

\subsection{Notation}

We first introduce some notation that helps keep track of the positions of the expected vortex
centers. We typically write $a = (a_1,\ldots,a_N) \in \Omega^N$ for these positions.
To control collision and escape to the boundary, we define a minimal measure of intervortex and vortex-boundary distance: 
\[
\rho(a) = \min \left\{ {1\over2} \min_{m \neq n} \LV a_m - a_n \RV, \min_{n = 1, \ldots, N} \operatorname{dist}(a_n, \p \Omega) \right\}.
\]
As we expect the energy to concentrate near the vortices, we sometimes need to cut out
vortex balls of radius $r$ or merely the vortex centers. We define
\[
\Omega_r(a) := \Omega \backslash \bigcup_{n = 1}^N \overline{B_r(a_n)}, \quad \Omega_0(a) := \Omega \backslash \{ a_1, \ldots, a_N\}.
\]

We denote $\LA \cdot, \cdot \RA$ to be the scalar product for elements in $\R^3$, and let $\LC \cdot, \cdot \RC$ denote
the $\R^2$ scalar product, which will arise often when we project $\R^3$ onto $\R^2 \times \{0\}$.  

\subsection{Energy density, vorticity, and the planar Jacobian}
The analogous quantity of  $\omega(\m)$ for $u: \R^2 \to \C$ is the Jacobian $J(u)$.
The Jacobian is useful for describing the winding number about a vortex in Ginzburg-Landau theory.  
We can also write  $J(u) = {1\over2} \curl j(u)$, where the supercurrent $j(u)$ is defined as
\[
j(u) = (i u, \nabla u), 
\]
for the complex-inner product $( \cdot, \cdot )$.  The connections between $\omega(\m)$ and $J(m)$, for the projection of $\m$ onto $\R^2$ (and complexified)  can be seen by:
\begin{align*}
J(m) & = m_3 \omega(\m) \\
\omega(\m) & = 3 m_3 J(m) + \curl \LC m_2 m_3 \nabla m_1 - m_1 m_3 \nabla m_2 \RC,
\end{align*}
both of which follow from direct calculations \cite[Lemma 1]{Kurzke-Melcher-Moser:11}.

Our method of proof entails tracking how concentrations in the energy density and in the vorticity move in time.  In that regard we write down differential identities that are satisfied by $e_\e(\m)$ and $\omega(\m)$. 
In particular we have:
\begin{equation}  \label{conservationenergy}
{\p \over \p t} e_\e(\m) = \alpha_\e \LA \f_\e (\m) , \nabla \m \RA - \operatorname{div} \LA \m \times \f_\e (\m) , \nabla \m \RA - \alpha_\e \LV \f_\e (\m) \RV^2
\end{equation}
and   
\begin{equation} \label{conservationvorticity}
{\p \over \p t} \omega (\m) = \curl \LA \f_\e(\m) , \nabla \m \RA 
+ \alpha_\e \curl \LA \m \times \f_\e (\m) , \nabla \m \RA,
\end{equation}
where $\f_\e$ is defined in \eqref{L2-gradient}.  See \cite{Kurzke-Melcher-Moser-Spirn:11} for derivations of these differential identities.

\subsection{Renormalized energy}

The micromagnetic energy $E_\e(\m)$ can be expanded up to second order.  This expansion, presented in Hang-Lin \cite{Hang-Lin:01}, says that the micromagnetic energy behaves
roughly like
\[
E_\e(\m) \approx N(\pi \logep+\gamma) + W(a, d), 
\]
where $W(a,d)$ is the renormalized energy of Bethuel-Brezis-Helein \cite{Bethuel-Brezis-Helein:94}.

To define $W(a,d)$ we first write down the \emph{canonical harmonic map} $m_*: \Omega \to \mathbb{S}^1$, which is defined via the Hodge system
\[
\operatorname{div} j(m_*) = 0 \qquad \hbox{ and } \qquad \curl j(m_*) = \sum_{n=1}^N 2 \pi \delta_{a_n}
\]
with either the Dirichlet  condition
\[
m_* = g
\]
on $\p \Omega$ or the Neumann condition
\[
j(m_*)  \cdot \nu  = 0
\]
on $\p \Omega$, depending on the problem that we study.  
Then $m_*$ can be written as 
\[
m_* = e^{i\theta}\prod_{n=1}^N \left(\frac{ x - a_n }{ |x - a_n|}\right)^{d_n}
\]
for a harmonic function $\theta : \Omega \to \R$ satisfying the appropriate boundary conditions.
Then
\begin{align*}
W(a,d) & = \lim_{\rho \to 0} \LB \int_{\Omega_\rho(a)} {1\over2} \LV \nabla m_* \RV^2 - N \pi \log {1\over\rho} \RB \\
& = - \pi \sum_{m \neq n} d_m d_n \log |a_m - a_n|  + \hbox{boundary effects},
\end{align*}
which is precisely the Coulombic energy of point particles (with a repulsive boundary in the case of a Dirichlet boundary condition).  

Finally, we note that if $\phi \in C_0^\infty(\Omega)$ such that $\nabla^\perp \nabla \phi$ vanishes near a vortex, then we have \cite{Bethuel-Brezis-Helein:94} 
\begin{equation} \label{gradientRenormalized}
\pi \sum_{n=1}^N \nabla^\perp \phi(a_n) \cdot  {\p \over \p a_n} W(a,d) = - \int_\Omega \nabla^\perp \nabla \phi : \LC \nabla m_* \otimes 
\nabla m_* \RC dx.
\end{equation}
This identity will be important for the derivation of the motion law of the vortices.

\section{The Landau-Lifshitz-Gilbert equation} \label{sectionLLG}

The proof of Theorem \ref{vortexmotionlawLLG} requires a combination of ideas
from various theories and spans several sections. Most of the
arguments do not make use of the boundary conditions at all; therefore, we
treat Dirichlet and Neumann data simultaneously.

In this section, we discuss properties of the Landau-Lifshitz-Gilbert equation
that can be described as well-known, even though they may be difficult to find
in the literature for the exact version of the equation used here. These properties
have been derived for related equations---in particular the harmonic map heat flow
\cite{Struwe:85,Chang:89,Ding-Tian:95,Qing:95} and a simpler version of the Landau-Lifshitz-Gilbert equation
\cite{Guo-Hong:93}. Examining the corresponding arguments, we see that they carry
over to our situation.

We fix $\epsilon > 0$ for the moment. Suppose that we have initial data $\m_0 \in C^\infty(\overline{\Omega};\mathbb{S}^2)$,
and we want to solve the Landau-Lifshitz-Gilbert equation under either Dirichlet or Neumann
boundary conditions.
Following Struwe \cite{Struwe:85}, Chang \cite{Chang:89}, and Guo-Hong \cite{Guo-Hong:93},
we can construct a weak solution $\m$ with $\m(0) = \m_0$ that is smooth away from finitely many points. 
That is, there exists a finite set of singular points
$\Sigma = \{(t^1,x^1),\ldots,(t^I,x^I)\} \subset (0,\infty) \times \overline{\Omega}$ such that
$\m \in C^\infty(([0,\infty) \times \overline{\Omega}) \backslash \Sigma;\mathbb{S}^2)$.
Furthermore, the map $\m$ belongs to the Sobolev space
\[
H^1((0,T);L^2(\Omega;\mathbb{S}^2)) \cap L^\infty((0,\infty);H^1(\Omega;\mathbb{S}^2))
\]
for every $T > 0$.

The singularities arising in this construction are a consequence of energy concentration.
If $(t^i,x^i)$ is a point in $\Sigma$ and $r > 0$ is small enough to separate this point
from all other singularities, then we have a $q_i \in \frac{1}{2}\Z$ such
that
\begin{align*}
\int_{\{t^i\} \times B_r(x^i)} e_\e(\m) \, dx + 4 \pi \LV q_i \RV & \leq \liminf_{t \nearrow t^i} \int_{\{t\} \times B_r(x^i)} e_\e(\m) \, dx, \\
\int_{\{t^i\} \times B_r(x^i)} \omega(\m) \, dx + 4 \pi  q_i  & = \lim_{t \nearrow t^i} \int_{\{t\} \times B_r(x^i)} \omega(\m) \, dx.
\end{align*}
Indeed, the number $q_i$ has a geometric interpretation, and in order to understand it,
we need to analyze the process leading to singularities. This is done with the tools
developed by Ding-Tian \cite{Ding-Tian:95} and Qing \cite{Qing:95}.

Fixing the singular point $(t^i,x^i)$, we can find sequences $x_j \to x^i$, $t_j \nearrow t^i$, and
$r_j \searrow 0$, such that the rescaled maps $\m_j(x) = \m(t_j,r_j x + x_j)$ converge to a critical point
of the Dirichlet functional. These critical points are called harmonic maps, and if they
arise through a blow-up process as here, we also speak of harmonic bubbles. Via the stereographic
projection, every harmonic bubble at a singularity in the interior of $\Omega$ is
identified with a harmonic map $\mathbb{S}^2 \to \mathbb{S}^2$. Since such harmonic maps are very well understood \cite{Eells-Lemaire:78}, we have
good information about their energies. If $\tilde{q}$ is the topological degree
of a bubble at a singularity in the interior of $\Omega$, then the corresponding energy is
$4\pi|\tilde{q}|$. Under Dirichlet boundary conditions, this is the only situation
possible. Under Neumann conditions, bubbles at the boundary are also conceivable. They
give rise to harmonic maps on a hemisphere (with Neumann boundary conditions on the equator),
which can be extended to harmonic maps on $\mathbb{S}^2$ by reflection. Thus we can think of them as
half-bubbles, and the degree is naturally measured by a half-integer $\tilde{q}$.
With this convention, the energy of a bubble remains $4\pi|\tilde{q}|$.

The number $q_i$ in the preceding formulas is the sum of the degrees of one or several bubbles
at the singularity $(t^i,x^i)$. In the first formula, we have an inequality only, because degrees can cancel
each other, whereas the energy is always positive. Equality holds if all the degrees
have the same sign. In any case, we have at least one bubble and therefore
\[
\int_{\{t^i\} \times B_r(x^i)} e_\e(\m) \, dx + 2 \pi \leq \liminf_{t \nearrow t^i} \int_{\{t\} \times B_r(x^i)} e_\e(\m) \, dx.
\]
Using the formula at every singular time and the energy conservation law \eqref{conservationenergy} between
singular times, we obtain the energy inequality
\begin{equation} \label{energy_inequality}
E_\epsilon(\m(t_2)) + 2\pi |\Sigma \cap ((t_1,t_2] \times \overline{\Omega})| + \alpha_\epsilon \int_{t_1}^{t_2} \int_\Omega \left|\dd{\m}{t}\right|^2 \, dx \, dt \le E_\epsilon(\m(t_1))
\end{equation}
for $0 \le t_1 \le t_2$.

Weak solutions of the Landau-Lifshitz-Gilbert equation are not necessarily
unique \cite{Bertsch-DalPasso-vanderHout:02,Topping:02}, but uniqueness
follows when we impose suitable additional conditions. In particular,
it follows from arguments of Freire \cite{Freire:95} (see also Harpes \cite{Harpes:04}
and Rupflin \cite{Rupflin:08}) that there is exactly one weak solution such that
the energy is non-increasing in time. 
We call it the \emph{energy decreasing solution}.

We now consider initial data
$\m_\epsilon^0 \in C^\infty(\overline{\Omega};\mathbb{S}^2)$ that satisfy the
assumptions in section \ref{results}.
We study the energy decreasing solutions $\m_\epsilon$ of the Landau-Lifshitz-Gilbert equation \eqref{LLG}
in $(0,\infty) \times \Omega$ with $\m_\epsilon(0) = \m_\epsilon^0$. Then
each $\m_\epsilon$ is smooth away from finitely many singular times.
That is, there exists a finite set $\Sigma_\epsilon \subset (0,\infty)$
such that $\m_\epsilon$ is smooth in $([0,\infty) \backslash \Sigma_\epsilon) \times \overline{\Omega}$.
Furthermore, we have the energy inequality
\[
E_\epsilon(\m_\epsilon(t_2)) + 2\pi |\Sigma_\epsilon \cap (t_1,t_2]| + \alpha_\epsilon \int_{t_1}^{t_2} \int_\Omega \left|\dd{\m_\epsilon}{t}\right|^2 \, dx \, dt \le E_\epsilon(\m(t_1))
\]
for all $t_1,t_2 \in [0,\infty)$ with $t_1 < t_2$.
Thus we have derived the first statement of Theorem \ref{vortexmotionlawLLG}
from known arguments. It is more difficult to prove the motion law
for the vortices, although in the first step, we still rely
mostly on known results and ideas.

\section{Continuous vortex motion}

The aim of this section is to show that the vortices
(centered at $a_1^0,\ldots,a_N^0$ at time $0$) persist
for positive times and the vortex centers have a continuous
trajectory. More precisely, we want to prove the following.

\begin{theorem} \label{thm:vortex_motion}
There exist a sequence $\epsilon_k \searrow 0$, a number $T > 0$, and a continuous
curve $a : [0,T) \to \Omega^N$ with $a(0) = a^0$, such that $\rho(a(t)) > 0$ and
\[
\alpha_{\epsilon_k} e_{\epsilon_k}(\m_{\epsilon_k}(t)) \to \pi \sum_{n = 1}^N \delta_{a_n(t)} \quad \text{and} \quad J(m_{\epsilon_k}(t)) \to \pi \sum_{n = 1}^N d_n \delta_{a_n(t)}
\]
as distributions for every $t \in [0,T)$ as $k \to \infty$.

Moreover, there exists a finite set $\Sigma \subset [0,T]$ with the
following property. Suppose that $[t_1,t_2] \subset [0,T]$ is an
interval with $\Sigma \cap [t_1,t_2] = \emptyset$. Then there is
a number $K \in \N$ such that $\m_{\epsilon_k}$ is smooth in $(t_1,t_2) \times \Omega$
for every $k \ge K$.
\end{theorem}

\begin{proof}
The convergence of the Jacobians $J(m_{\epsilon_k})$ (for a suitable
subsequence) follows almost directly from results of Sandier-Serfaty
\cite{Sandier-Serfaty:04.1}. Namely, Theorem 3 implies that
there exist a sequence $\epsilon_k \searrow 0$ and a function
$J_* \in C^{0,1/2}([0,T);\mathcal{M}(\Omega))$ such that
\[
J(m_{\epsilon_k}(t)) \to J_*(t)
\]
in $\dot{W}^{-1,1}(\Omega)$ for every $t \in [0,T)$.
On the other hand, the logarithmic energy bound obtained
from the energy inequality \eqref{energy_inequality}, together with results of
Jerrard-Soner \cite[Theorem 3.1]{Jerrard-Soner:02},
give good information about the structure of $J_*(t)$: it must
be of the form
\[
J_*(t) = \pi \sum_{n = 1}^{\tilde{N}} \tilde{d}_n(t) \delta_{\tilde{a}_n(t)}
\]
for certain integers $\tilde{d}_1(t),\ldots,\tilde{d}_{\tilde{N}}(t)$
and certain points $\tilde{a}_1(t),\ldots,\tilde{a}_{\tilde{N}}(t) \in \Omega$.
Let $\eta > 0$. By the continuity of $J_*$, we can choose $T$ small enough
so that
\[
\left\|\sum_{n = 1}^{\tilde{N}} \tilde{d}_n(t) \delta_{\tilde{a}_n(t)} - \sum_{n = 1}^N d_n \delta_{a_n^0}\right\|_{\dot{W}^{-1,1}(\Omega)} \le \eta.
\]
If $\eta$ is chosen sufficiently small, then it follows that
$\tilde{N} = N$ and $\tilde{d}_n(t) = d_n$ for $n = 1,\ldots,N$.
Moreover, the continuity of $J_*$ also implies that
the curves $a_1,\ldots,a_N$ are continuous in $[0,T)$.

The convergence of the Jacobians implies that for every
$t \in [0,T)$ and for every $r \in (0,\rho(a(t))]$, we have
\[
\int_{B_r(a_n(t))} e_{\epsilon_k}(\m_{\epsilon_k}) \, dx \ge \pi \log \frac{r}{\epsilon_k} - C.
\]
Hence
\[
\alpha_{\epsilon_k} \int_{\Omega_r(a(t))} e_{\epsilon_k}(\m_{\epsilon_k}) \, dx \to 0,
\]
and the convergence of the energy density follows as well.

For every $t \in [0,T)$, we have
\[
\liminf_{k \to \infty} \left(E_{\epsilon_k}(\m_{\epsilon_k}(t)) - N\pi \log \frac{1}{\epsilon_k}\right) > - \infty.
\]
We may assume that the same holds at $t = T$ (otherwise
we replace $T$ by a smaller number). Then we have
a uniform bound for the number of singular times in $[0,T]$ by \eqref{energy_inequality}, and we may
assume that there exist a finite set $\Sigma \subset [0,T]$
such that
\[
\limsup_{k \to \infty} \max_{s \in \Sigma_{\epsilon_k} \cap [0,T]} \min_{t \in \Sigma} |s - t| = 0.
\]
This set then has the required property.
\end{proof}

\section{Strong convergence}

In this section, we study the behavior of the solutions at a
fixed time where no concentration of the kinetic energy occurs.
It is then convenient to rescale the time axis,
and therefore we now consider the family of equations
\begin{equation} \label{LLG_rescaled}
\frac{1}{\alpha_\epsilon} \dd{\m_\epsilon}{t} = \m \times \left(\dd{\m_\epsilon}{t} - \f_\epsilon(\m_\epsilon)\right) \quad \mbox{in } (0,1) \times \Omega.
\end{equation}
We assume that we have weak solutions $\m_\epsilon \in C^\infty([0,1) \times \Omega;\mathbb{S}^2)$
satisfying
\begin{equation} \label{time_derivative}
\lim_{\epsilon \to 0} \int_0^1 \int_\Omega \left|\dd{\m_\epsilon}{t}\right|^2 \, dx \, dt = 0.
\end{equation}
Let $a \in \Omega^N$ and suppose that there exists a function $\gamma : (0,\infty) \to \R$
such that for any $r > 0$,
\[
\limsup_{\epsilon \to 0} \sup_{0 \le t \le 1} \int_{\Omega_r(a)} e_\epsilon(\m_\epsilon(t,x)) \, dx \le \gamma(r).
\]
Furthermore, we assume that $\m_\epsilon \rightharpoonup \m_*$
weakly in $H_\loc^1((0,1) \times \Omega_0(a))$
and strongly in $L_\loc^2((0,1) \times \Omega_0(a))$.

\begin{lemma} \label{lemma:strong_convergence}
Under these assumptions the following holds true.
\begin{enumerate}
\item \label{convergence_part1} The limit is a time-independent harmonic map into $S^1 \times \{0\}$, i.e., it
satisfies $m_{*3} = 0$ and $\dd{\m_*}{t} = 0$, and
\[
\div j(m_*) = 0, \quad \curl j(m_*) = 0 \quad \mbox{in } (0,1) \times \Omega_0(a).
\]
\item Consider the measures
\[
\mu_\epsilon = \mathcal{L}^2 \restr (\nabla \m_\epsilon(0) \otimes \nabla \m_\epsilon(0)).
\]
Let $r > 0$. There exist a sequence $\epsilon_k \searrow 0$, finitely many points $b_1,\ldots,b_P \in \Omega_r(a)$, and
matrices $B_1,\ldots,B_P \in \R^{2 \times 2}$ such that
\[
\mu_{\epsilon_k} \stackrel{*}{\rightharpoonup} \mathcal{L}^2 \restr (\nabla \m_* \otimes \nabla \m_*) + \sum_{p = 1}^P B_p \delta_{b_p}
\]
weakly* in $(C_0^0(\Omega_r(a)))^*$. Furthermore, the number $P$
depends only on $\gamma(r)$.
\end{enumerate}
\end{lemma}

\begin{remark}
We will apply this result in a situation where the assumptions
are justified only for a subsequence $\epsilon_k \searrow 0$.
We give a proof only for the result as stated in order to
avoid awkward notation; however, everything works for a subsequence
as well.
\end{remark}

\begin{proof}
It follows immediately from the assumptions that $m_{*3} = 0$
and $\dd{\m_*}{t} = 0$. As a map with values in a $1$-dimensional
manifold, the limit automatically satisfies
\[
\curl j(m_*) = 0 \quad \mbox{in } (0,1) \times \Omega_0(a).
\]
Equation \eqref{LLG_rescaled} implies
\begin{equation} \label{third_component}
\div j(m_\epsilon) = -\frac{1}{\alpha_\epsilon} \dd{m_{\epsilon 3}}{t} + \left(im_\epsilon,\dd{m_\epsilon}{t}\right).
\end{equation}
Multiplying by a test function and integrating by parts in the first term on the right hand side,
we also find
\[
\div j(m_*) = 0 \quad \mbox{in } (0,1) \times \Omega_0(a).
\]
It follows by elliptic regularity theory that $\m_*$ is smooth away from the vortices.

Let $\delta > 0$ and suppose that we have a disk $B_{2r}(x_0) \subset \Omega$ with
\begin{equation} \label{small_energy}
\limsup_{\epsilon \to 0} \int_{B_{2r}(x_0)} e_\epsilon(\m_\epsilon(0)) \, dx \le \delta^2.
\end{equation}
We first want to show that a similar inequality persists for some positive time. To this end,
let $\zeta \in C_0^\infty(B_{2r}(x_0))$ with $0 \le \zeta \le 1$ and $|\nabla \zeta| \le C/r$.
Here and in the following, we indiscriminately use the symbol $C$ for various universal constants. Set
\[
F_\epsilon(t) = \int_{\{t\} \times \Omega} \zeta^2 e_\epsilon(\m_\epsilon) \, dx \quad \mbox{and} \quad
G_\epsilon(t) = \int_{\{t\} \times \Omega} \left|\dd{\m_\epsilon}{t}\right|^2 \, dx.
\]
Then we have
\[
\begin{split}
F_\epsilon'(t) & = - \int_{\{t\} \times \Omega} \zeta^2 \left|\dd{\m_\epsilon}{t}\right|^2 \, dx - 2 \int_{\{t\} \times \Omega} \zeta \nabla \zeta \cdot \scp{\dd{\m_\epsilon}{t}}{\nabla \m_\epsilon} \, dx \\
& \le \frac{C}{r}\sqrt{F_\epsilon(t) G_\epsilon(t)}.
\end{split}
\]
It follows that
\[
\frac{d}{dt} \sqrt{F_\epsilon(t)} \le \frac{C}{r} \sqrt{G_\epsilon(t)}.
\]
Hence
\[
\sqrt{F_\epsilon(s)} \le C\delta + \frac{C}{r}\left(s \int_0^s \int_\Omega \left|\dd{\m_\epsilon}{t}\right|^2 \, dx \,  dt\right)^{1/2},
\]
and in particular
\[
\int_{B_r(x_0)} e_\epsilon(\m_\epsilon(t,x)) \, dx \le C\delta^2
\]
for $t \in [0,r^2 \delta^2]$. Set $\tau = r^2 \delta^2$.

Choose a cut-off function $\eta \in C_0^\infty(B_r(x_0))$ with $0 \le \eta \le 1$ and $|\nabla \eta| \le C/r$. Consider the Hodge decomposition
\[
\eta(j(m_\epsilon(t)) - j(m_*)) = \nabla \phi_\epsilon + \nabla^\perp \psi_\epsilon \quad \mbox{in } \R^2,
\]
then
\[
\|\nabla \phi_\epsilon(t)\|_{L^2(\R^2)}^2 + \|\nabla \psi_\epsilon(t)\|_{L^2(\R^2)}^2 = \|\eta(j(m_\epsilon(t)) - j(m_*))\|_{L^2(B_r(x_0))}^2
\]
for every $t \in [0,\tau]$. Using \eqref{third_component}, we also obtain
\[
\begin{split}
\int_0^{\tau} \int_\Omega \eta^2 |j(m_\epsilon) - j(m_*)|^2 \, dx \, dt
& = \frac{1}{\alpha_\epsilon} \int_0^{\tau} \int_\Omega \eta \phi_\epsilon \dd{m_{\epsilon 3}}{t} \, dx \, dt \\
& \quad - \int_0^{\tau} \int_\Omega \eta \phi_\epsilon \left(im_\epsilon,\dd{m_\epsilon}{t}\right) \, dx \, dt \\
& \quad - \int_0^{\tau} \int_\Omega \phi_\epsilon \nabla \eta \cdot (j(m_\epsilon) - j(m_*)) \, dx \, dt \\
& \quad + \int_0^{\tau} \int_\Omega \eta \nabla^\perp \psi_\epsilon \cdot (j(m_\epsilon) - j(m_*)) \, dx \, dt.
\end{split}
\]
For almost every $t \in (0,\tau)$, we have $\phi_\epsilon(t) \to 0$ strongly in $L_\loc^p(\R^2)$ for every $p < \infty$.
Using Lebesgue's convergence theorem, we prove that $\phi_\epsilon(t) \to 0$ strongly in $L_\loc^p((0,\tau) \times \R^2)$.
Hence
\[
\int_0^{\tau} \int_\Omega \phi_\epsilon \nabla \eta \cdot (j(m_\epsilon) - j(m_*)) \, dx \, dt \to 0
\]
and
\[
\int_0^{\tau} \int_\Omega \eta \phi_\epsilon \left(im_\epsilon,\dd{m_\epsilon}{t}\right) \, dx \, dt \to 0
\]
as $\epsilon \to 0$. Moreover,
\begin{multline*}
\int_0^{\tau} \int_\Omega \eta \nabla^\perp \psi_\epsilon \cdot (j(m_\epsilon) - j(m_*)) \, dx \, dt \\
\begin{aligned}
& = \int_0^{\tau} \int_\Omega \eta \nabla^\perp \psi_\epsilon \cdot (im_\epsilon - im_*,\nabla m_\epsilon) \, dx \, dt \\
& \quad + \int_0^{\tau} \int_\Omega \nabla^\perp \psi_\epsilon \cdot (im_*,\nabla(\eta(m_\epsilon - m_*))) \, dx \, dt \\
& \quad - \int_0^{\tau} \int_\Omega \nabla^\perp \psi \cdot \nabla \eta (im_*,m_\epsilon - m_*) \, dx \, dt.
\end{aligned}
\end{multline*}
We clearly have
\[
\int_0^{\tau} \int_\Omega \nabla^\perp \psi \cdot \nabla \eta (im_*,m_\epsilon - m_*) \, dx \, dt \to 0.
\]
For almost every $t$, we observe that
\[
[\eta(im_\epsilon - im_*)]_{\BMO(\R^2)} \le C \|\eta(\nabla m_\epsilon - \nabla m_*)\|_{L^2(\Omega)} + \frac{C}{r}\|m_\epsilon - m_*\|_{L^2(B_r(x_0))}.
\]
The last term on the right hand side converges to $0$. With compensated compactness arguments \cite{Coifman-Lions-Meyer-Semmes}, we now estimate
\[
\int_0^{\tau} \int_\Omega \eta \nabla^\perp \psi_\epsilon \cdot (im_\epsilon - im_*,\nabla m_\epsilon) \, dx \, dt \le C\delta \int_0^{\tau} \int_\Omega \eta^2 |\nabla m_\epsilon - \nabla m_*|^2 \, dx \, dt + o_\epsilon(1).
\]
Similarly,
\[
[im_*]_{\BMO(B_r(x_0))} \le C \delta,
\]
and therefore,
\[
\int_0^{\tau} \int_\Omega \nabla^\perp \psi_\epsilon \cdot (im_*,\nabla(\eta(m_\epsilon - m_*))) \, dx \, dt \le C\delta \int_0^{\tau} \int_\Omega \eta^2 |\nabla m_\epsilon - \nabla m_*|^2 \, dx \, dt + o_\epsilon(1).
\]
Finally, we compute
\begin{multline*}
\int_0^{\tau} \int_\Omega \eta \phi_\epsilon \dd{m_{\epsilon 3}}{t} \, dx \, dt = \int_{\{\tau\} \times \Omega} \eta \phi_\epsilon m_{\epsilon 3} \, dx - \int_{\{0\} \times \Omega} \eta \phi_\epsilon m_{\epsilon 3} \, dx \\
- \int_0^{\tau} \int_\Omega \eta \dd{\phi_\epsilon}{t} m_{\epsilon 3} \, dx \, dt.
\end{multline*}
The first two terms on the right hand side are of order $O(\epsilon)$. Now we use the identity
\[
\dd{}{t} j(m_\epsilon) = 2\left(i\dd{m_\epsilon}{t},\nabla m_\epsilon\right) + \nabla \left(im_\epsilon,\dd{m_\epsilon}{t}\right);
\]
hence,
\[
\Delta \left(\dd{\phi_\epsilon}{t} - \eta\left(im_\epsilon,\dd{m_\epsilon}{t}\right)\right) = \div \left(2\eta \left(i\dd{m_\epsilon}{t},\nabla m_\epsilon\right) - \nabla \eta \left(im_\epsilon,\dd{m_\epsilon}{t}\right)\right).
\]
We conclude that
\[
\left\|\dd{\phi_\epsilon}{t}\right\|_{L^p((0,\tau) \times \R^2)} \le C
\]
for every $p < 2$. Thus
\[
\frac{1}{\alpha_\epsilon} \int_0^{\tau} \int_\Omega \eta \phi_\epsilon \dd{m_{\epsilon 3}}{t} \, dx \, dt \to 0.
\]
Combining these estimates, we obtain
\[
\int_0^{\tau} \int_\Omega \eta^2 |j(m_\epsilon) - j(m_*)|^2 \, dx \, dt \le C\delta \int_0^{\tau} \int_\Omega \eta^2 |\nabla m_\epsilon - \nabla m_*|^2 \, dx \, dt + o_\epsilon(1).
\]

Now let
\[
\Lambda_\epsilon = \set{t \in (0,\tau)}{\left\|\dd{m_\epsilon}{t}(t)\right\|_{L^2(\Omega)} \le \frac{\alpha_\epsilon^2}{\epsilon}}.
\]
Fix $\epsilon \in (0,\frac{1}{2}]$ and $t \in \Lambda_\epsilon$. For a given $x_1 \in B_r(x_0)$, consider
\[
\tilde{\m}(x) = \m_\epsilon(t,\epsilon x + x_1).
\]
Then we have
\[
\|\f_1(\tilde{\m})\|_{L^2(B_1(0))} \le C\alpha_\epsilon
\]
and
\[
\|e_1(\tilde{\m})\|_{L^1(B_1(0))} \le C \delta^2.
\]
Choose $\beta > 0$. If $\delta$ and $\epsilon$ are sufficiently small, then known estimates
from the theory of harmonic maps (e.g. \cite[Lemma 2.1]{Ding-Tian:95}) imply $\tilde{m}_3^2 \le \beta$.
We conclude that $m_{\epsilon 3}^2 \le \beta$ in $\Lambda_\epsilon \times B_r(x_0)$.

Using the equation
\[
\Delta m_{\epsilon 3} = - |\nabla \m_\epsilon|^2 m_{\epsilon 3} + \frac{1}{\epsilon^2} (m_{\epsilon 3} - m_{\epsilon 3}^3) + f_{\epsilon 3}(\m_\epsilon),
\]
we obtain
\begin{multline*}
\int_{\{t\} \times \Omega} \eta^2 |\nabla m_{\epsilon 3}|^2 \, dx = \int_{\{t\} \times \Omega} \eta^2 m_{\epsilon 3}^2 |\nabla \m_\epsilon|^2 \, dx - \frac{1}{\epsilon^2} \int_{\{t\} \times \Omega} \eta^2 (m_{\epsilon 3}^2 - m_{\epsilon 3}^4) \, dx \\
- \int_{\{t\} \times \Omega} \eta^2 m_{\epsilon 3} f_{\epsilon 3}(\m_\epsilon) \, dx - 2 \int_{\{t\} \times \Omega} \eta m_{\epsilon 3} \nabla \eta \cdot \nabla m_{\epsilon 3} \, dx.
\end{multline*}
For $t \in \Lambda_\epsilon$, we have
\[
\begin{split}
\int_{\{t\} \times \Omega} \eta^2 m_{\epsilon 3}^2 |\nabla \m_\epsilon|^2 \, dx 
& \le 2  \int_{\{t\} \times \Omega} \eta^2 m_{\epsilon 3}^2 |\nabla \m_\epsilon - \nabla \m_*|^2 \, dx \\
& \quad + 2\int_{\{t\} \times \Omega} \eta^2 m_{\epsilon 3}^2 |\nabla m_*|^2 \, dx \\
& \le 2\beta \int_{\{t\} \times \Omega} \eta^2 |\nabla \m_\epsilon - \nabla \m_*|^2 \, dx + o_\epsilon(1)
\end{split}
\]
and $m_{\epsilon 3}^2 - m_{\epsilon 3}^4 \ge \frac{3}{4} m_{\epsilon 3}^2$ in $\{t\} \times B_r(x_0)$.
Furthermore,
\[
\int_{\{t\} \times \Omega} \eta^2 m_{\epsilon 3} f_{\epsilon 3}(\m_\epsilon) \, dx \le C\alpha_\epsilon
\]
and
\[
- 2 \int_{\{t\} \times \Omega} \eta m_{\epsilon 3} \nabla \eta \cdot \nabla m_{\epsilon 3} \, dx \le C\epsilon \log \frac{1}{\epsilon}.
\]
Note also that
\[
\int_{(0,\tau) \backslash \Lambda_\epsilon} \int_\Omega \eta^2 e_\epsilon(\m_\epsilon) \, dx \, dt \le C\epsilon^2 \left(\log \frac{1}{\epsilon}\right)^4.
\]
It follows that
\[
\int_0^{\tau} \int_\Omega \eta^2 \left(|\nabla m_{\epsilon 3}|^2 + \frac{m_{\epsilon 3}^2}{\epsilon^2}\right) \, dx \, dt \le C\beta \int_0^{\tau} \int_\Omega \eta^2 |\nabla \m_\epsilon - \nabla \m_*|^2 \, dx \, dt + o_\epsilon(1).
\]

Combined with the previous estimate, this yields
\begin{multline*}
\int_0^{\tau} \int_\Omega \eta^2 \left(|j(m_\epsilon) - j(m_*)|^2 + |\nabla m_{\epsilon 3}|^2 + \frac{m_{\epsilon 3}^2}{\epsilon^2}\right) \, dx \, dt \\
\le C(\beta + \delta) \int_0^{\tau} \int_\Omega \eta^2 |\nabla \m_\epsilon - \nabla \m_*|^2 \, dx \, dt + o_\epsilon(1).
\end{multline*}
Since
\begin{multline*}
\nabla m_\epsilon - \nabla m_* = \frac{m_\epsilon}{|m_\epsilon|} \nabla \sqrt{1 - m_{\epsilon 3}^2} + \frac{im_\epsilon}{|m_\epsilon|^2}(j(m_\epsilon) - j(m_*)) \\
+ i \left(\frac{m_\epsilon}{|m_\epsilon|^2} - m_*\right)j(m_*),
\end{multline*}
we have
\begin{multline*}
\int_{\Lambda_\epsilon} \int_\Omega \eta^2 |\nabla m_\epsilon - \nabla m_*|^2 \, dx \\
\le C\int_0^{\tau} \int_\Omega \eta^2 \left(|j(m_\epsilon) - j(m_*)|^2 + |\nabla m_{\epsilon 3}|^2\right) \, dx \, dt + o_\epsilon(1).
\end{multline*}
For the integral over $(0,\tau) \backslash \Lambda_\epsilon$, we have the same kind of estimate as before. Thus
\begin{multline*}
\int_0^{\tau} \int_\Omega \eta^2 \left(|\nabla \m_\epsilon - \nabla \m_*|^2 + \frac{m_{\epsilon 3}^2}{\epsilon^2}\right) \, dx \, dt \\
\le C(\beta + \delta) \int_0^{\tau} \int_\Omega \eta^2 |\nabla \m_\epsilon - \nabla \m_*|^2 \, dx \, dt + o_\epsilon(1).
\end{multline*}
If $\beta$ and $\delta$ are chosen sufficiently small, then it follows that $\nabla \m_\epsilon \to \nabla \m_*$ and $m_{\epsilon 3}/\epsilon \to 0$ strongly in $L^2((0,\tau) \times B_{r/2}(x_0))$.

Finally, let $\chi \in C_0^\infty(B_{r/2}(x_0))$ and consider again the energy identity from \eqref{conservationenergy}
\begin{multline*}
\frac{d}{dt} \int_{\{t\} \times \Omega} \chi^2 e_\epsilon(\m_\epsilon) \, dx = -\int_{\{t\} \times \Omega} \chi^2 \left|\dd{\m_\epsilon}{t}\right|^2 \, dx \\
- 2 \int_{\{t\} \times \Omega} \chi \nabla \chi \cdot \scp{\dd{\m_\epsilon}{t}}{\nabla \m_\epsilon} \, dx.
\end{multline*}
Integrating in time and using \eqref{time_derivative}, we see that
\[
\osc_{0 \le t \le \tau} \int_{\{t\} \times \Omega} \chi^2 e_\epsilon(\m_\epsilon) \, dx \to 0 \quad \mbox{as } \epsilon \to 0.
\]
Hence we have
\[
\limsup_{\epsilon \to 0} \int_\Omega \chi^2 e_\epsilon(\m_\epsilon(0)) \, dx \le \frac{1}{2} \int_\Omega \chi^2 |\nabla \m_*|^2 \, dx.
\]
That is, we have $\m_\epsilon(0) \to \m_*$ strongly in $H^1(B_{r/4}(x_0))$ and $m_{\epsilon 3}(0)/\epsilon \to 0$ strongly in
$L^2(B_{r/4}(x_0))$---assuming the smallness condition \eqref{small_energy}.

The second statement of the lemma now follows by a standard covering argument.
\end{proof}

\section{The motion law}

We consider the solutions $\m_\epsilon$ of \eqref{LLG} constructed in section \ref{sectionLLG} again.

Let $T > 0$ be the number from Theorem \ref{thm:vortex_motion}.
We consider the sequence $\epsilon_k \searrow 0$ from the same
theorem and we use the same notation again. In particular $\Sigma \subset [0,T]$ is the
accumulation set of all the singular times of $\m_{\epsilon_k}$ for $k \to \infty$.
As $\Sigma$ is finite and the vortex trajectories are continuous, it suffices to verify
the motion law away from $\Sigma$.
Let $\m_* = (m_*,0) : (0,T) \times \Omega \to \mathbb{S}^2$ be the map such
that $m_*(t)$ is the canonical harmonic map with vortices of degree $d_n$ at $a_n(t)$, $n = 1,\ldots,N$. Note that we may assume
$\m_{\epsilon_k} \rightharpoonup \m_*$ locally in $H^1$ away
from the vortices by part (\ref{convergence_part1}) of Lemma \ref{lemma:strong_convergence}
(applied to time-rescaled versions of $\m_\epsilon$).

Consider a time interval $(t_1,t_2)$ with $\Sigma \cap [t_1,t_2] = \emptyset$,
and so short that for some $r \in (0,\rho(a(t_1))]$, we have
$a_n(t) \in B_{r/2}(a_n(t_1))$ for all $t \in [t_1,t_2]$ and $n = 1,\ldots,N$.
Define the measures
\[
\bar{\mu}_\epsilon = \mathcal{L}^3 \restr (\nabla \m_\epsilon \otimes \nabla \m_\epsilon)
\]
on $[t_1,t_2] \times \Omega_r(a(t_1))$. There exists a subsequence such that we have the convergence
$\bar{\mu}_{\epsilon_k} \stackrel{*}{\rightharpoonup} \bar{\mu}$ for a matrix valued Radon measure $\bar{\mu}$ on $[t_1,t_2] \times \Omega_r(a(t_1))$.
For almost every $t \in (t_1,t_2)$, there exist a measure $\mu(t)$ on $\Omega_r(a(t_1))$ such that
for all $\eta \in C_0^0(\Omega_r(a(t_1)))$,
\begin{equation} \label{Lebesgue_point}
\int_{\Omega_r(a(t_1))} \eta \, d\mu(t) = \lim_{h \searrow 0} \frac{1}{2h} \int_{(t - h,t + h) \times \Omega_r(a(t_1))} \eta \, d\bar{\mu}.
\end{equation}
Moreover, for all $\xi \in C_0^0([t_1,t_2] \times \Omega_r(a(t_1)))$, we have
\[
\int_{[t_1,t_2] \times \Omega_r(a(t_1))} \xi \, d\bar{\mu} = \int_{t_1}^{t_2} \int_{\Omega_r(a(t_1))} \xi \, d\mu(t) \, dt.
\]
We also consider the measures
\[
\mu_*(t) = \mathcal{L}^2 \restr (\nabla \m_*(t) \otimes \nabla \m_*(t)).
\]

\begin{lemma} \label{lemma:defect_measures}
There exist finitely many functions $b_1,\ldots,b_P : [t_1,t_2] \to \Omega_r(a(t_1))$ and
$B_1,\ldots,B_P : [t_1,t_2] \to \R^{2 \times 2}$ such that for almost every $t \in [t_1,t_2]$,
\[
\mu(t) = \mu_*(t) + \sum_{p = 1}^P B_p(t) \delta_{b_p(t)}.
\]
\end{lemma}

\begin{proof}
Fix $t \in (t_1,t_2)$ such that \eqref{Lebesgue_point} holds true and in addition,
\[
\lim_{h \searrow 0} \limsup_{k \to \infty} \left(\alpha_{\epsilon_k} \int_{t - h}^{t + h} \int_{\Omega_r(a(t_1))} \left|\dd{\m_{\epsilon_k}}{t}\right|^2 \, dx \, dt\right) = 0.
\]
We may assume without loss of generality that the latter is true for almost every $t$, for a subsequence with this
property exists.

Let $(\eta_\ell)_{\ell \in \N}$ be a dense sequence in $C_0^0(\Omega_r(a(t_1)))$. Choose a sequence
$h_k \searrow 0$, such that
\[
\left|\frac{1}{2h_k} \int_{[t - h_k,t + h_k] \times \Omega_r(a(t_1))} \eta_\ell \, d\bar{\mu} - \int_{\Omega_r(a(t_1))} \eta_\ell \, d\mu(t)\right| \le \frac{1}{k}
\]
for $\ell = 1,\ldots,k$. We can now replace $(\m_{\epsilon_k})_{k \in \N}$ by a subsequence
(but still use the same notation), such that
\[
\left|\frac{1}{2h_k} \int_{t - h_k}^{t + h_k} \int_{\Omega_r(a(t_1))} \eta_\ell \nabla \m_{\epsilon_k} \otimes \nabla \m_{\epsilon_k} \, dx \, dt - \int_{\Omega_r(a(t_1))} \eta_\ell \, d\mu(t)\right| \le \frac{2}{k}
\]
for $\ell = 1,\ldots,k$. Finally, by the continuity of the functions
\[
t \mapsto \int_{\Omega_r(a(t_1))} \eta_\ell \nabla \m_{\epsilon_k}(t) \otimes \nabla \m_{\epsilon_k}(t) \, dx,
\]
we can find $t_k \to t$ such that
\[
\left|\int_{\Omega_r(a(t_1))} \eta_\ell \nabla \m_{\epsilon_k}(t_k) \otimes \nabla \m_{\epsilon_k}(t_k) \, dx - \int_{\Omega_r(a(t_1))} \eta_\ell \, d\mu(t)\right| \le \frac{3}{k}
\]
for $\ell = 1,\ldots,k$.

Define
\[
\tilde{\m}_k(t,x) = \m_{\epsilon_k} (\alpha_{\epsilon_k}(t - t_k),x).
\]
Then $\tilde{\m}_k$ solves \eqref{LLG_rescaled} for $\epsilon_k$ instead of $\epsilon$, and the sequence
satisfies the counterpart to condition \eqref{time_derivative}. Using Lemma \ref{lemma:strong_convergence},
and passing to a subsequence again if necessary, we conclude that there exist $b_1(t),\ldots,b_P(t) \in \Omega_r(a(t_1))$
and $B_1(t),\ldots,B_P(t) \in \R^{2 \times 2}$, such that
\[
\int_{\Omega_r(a(t_1))} \eta_\ell \, d\mu(t) = \int_{\Omega_r(a(t_1))} \eta_\ell \, d\mu_*(t) + \sum_{p = 1}^P \eta_\ell(b_p(t)) B_p(t)
\]
for all $\ell \in \N$. Furthermore, the number $P$ is bounded by a constant independent of $t$
by the energy estimates away from the vortices. The claim now follows.
\end{proof}

Let $\phi,\psi \in C_0^\infty(\Omega)$. Then we compute, using \eqref{conservationenergy} and \eqref{conservationvorticity}, that
\begin{multline} \label{motion_law_positive_epsilon}
\frac{d}{dt} \int_{\{t\} \times \Omega} \left(\alpha_\epsilon \psi e_\epsilon(\m_\epsilon) + \phi \omega(\m_\epsilon)\right) \, dx \\
\begin{aligned}
& = - \alpha_\epsilon^2 \int_{\{t\} \times \Omega} \psi |\f_\epsilon(\m_\epsilon)|^2 \, dx \\
& \quad - \int_{\{t\} \times \Omega} (\alpha_\epsilon^2 \nabla \psi + \nabla^\perp \phi) \cdot \scp{\f_\epsilon(\m_\epsilon)}{\nabla \m_\epsilon} \,dx \\
& \quad + \alpha_\epsilon \int_{\{t\} \times \Omega} (\nabla \psi - \nabla^\perp \phi) \cdot
\scp{\m_\epsilon \times \f_\epsilon(\m_\epsilon)}{\nabla \m_\epsilon} \,dx
\end{aligned}
\end{multline}
in $[t_1,t_2]$. Note also that
\[
\int_{\{t\} \times \Omega} \nabla^\perp \phi \cdot \scp{\f_\epsilon(\m_\epsilon)}{\nabla \m_\epsilon} \, dx = - \int_{\{t\} \times \Omega} \nabla^\perp \nabla \phi : \nabla \m_\epsilon \otimes \nabla \m_\epsilon \, dx
\]
by an integration by parts.

First consider the identity with $\psi = 0$ and $\phi \in C_0^\infty(\Omega_r(a(t_1)))$. Define the measures
\[
\lambda_\epsilon(t) = \mathcal{L}^2 \restr \omega(\m_\epsilon).
\]
Then the functions
\[
t \mapsto \int_{\Omega_r(a(t_1))} \phi \, d\lambda_\epsilon(t)
\]
are uniformly Lipschitz continuous. It follows that there exists a sequence $\epsilon_k \searrow 0$ such that
$\lambda_{\epsilon_k}(t) \stackrel{*}{\rightharpoonup} \lambda(t)$ weakly* in $(C_0^0(\Omega_r(a(t_1))))^*$ for almost every $t$,
where $\lambda(t)$ is a Radon measure on $\Omega_r(a(t_1))$. Furthermore, we can choose this subsequence such that
the previous statements hold true. Lemma \ref{lemma:defect_measures} then implies that $\lambda(t)$ is of the
form
\[
\lambda(t) = \sum_{p = 1}^P \sigma_p(t) \delta_{b_p(t)}
\]
for certain functions $\sigma_1,\ldots,\sigma_P : [t_1,t_2] \to \R$.
Moreover, using Lemma 3 in \cite{Kurzke-Melcher-Moser-Spirn:11}, we see that
$\sigma_p(t) \in 4\pi\Z$ for every $p$ and almost every $t$.

Still using \eqref{motion_law_positive_epsilon} with $\psi = 0$ and $\phi \in C_0^\infty(\Omega_r(a(t_1)))$, we now obtain
\begin{multline*}
\sum_{p = 1}^P (\sigma_p(t_2) \phi(b_p(t_2)) - \sigma_p(t_1) \phi(b_p(t_1))) \\
= \int_{t_1}^{t_2} \left(\int_\Omega \nabla^\perp \nabla \phi : d\mu_*(t) + \sum_{p = 1}^P \nabla^\perp \nabla \phi(b_p(t)) : B_p(t)\right) \, dt.
\end{multline*}
Moreover,
\[
\int_{t_1}^{t_2} \int_\Omega \nabla^\perp \nabla \phi : d\mu_*(t) \, dt = 0,
\]
since $\m_*(t)$ is a harmonic map. It follows
that
\[
t \mapsto \sum_{p = 1}^P \sigma_p(t) \phi(b_p(t))
\]
belongs to $W^{1,\infty}(t_1,t_2)$ (so in particular it is continuous) and for almost every $t$,
\begin{equation} \label{*}
\frac{d}{dt} \sum_{p = 1}^P \sigma_p(t) \phi(b_p(t)) = \sum_{p = 1}^P \nabla^\perp \nabla \phi(b_p(t)) : B_p(t).
\end{equation}

We want to show that $\lambda(t)$ is constant, inserting test functions $\phi$ with
$\nabla^2 \phi(b_p(t)) = 0$. However, since
the $\sigma_p(t)$ may have a sign, the continuity of the above function
does not immediately rule out nucleation or annihilation of dipoles.
So we have to be a bit careful here.

Define $\#(t) = |\supp \lambda (t)|$. Choose $t_0 \in (t_1,t_2)$ with $\#(t_0) = \max_{t_1 < t < t_2} \#(t)$
(which exists because the function takes integer values).

\begin{lemma} \label{lemma1}
There exists an open interval $I$ with $t_0 \in I$, such that
$\#(t) \equiv \#(t_0)$ in $I$.
\end{lemma}

\begin{proof}
Choose $\rho > 0$ such that the balls $B_{2\rho}(x)$ for $x \in \supp \lambda(t_0)$
are contained in $\Omega$ and pairwise disjoint. Fix $x_0 \in \supp \lambda(t_0)$.
Let $\phi \in C_0^\infty(B_\rho(x_0))$ with $\phi(x_0) \not= 0$. Then if $I$ is
sufficiently short, we find
\[
\int_\Omega \phi \, d\lambda(t) \not= 0
\]
for $t \in I$. Hence $B_\rho(x_0)$ intersects $\supp \lambda(t)$.
It follows that $\#(t) \ge \#(t_0)$, and by the choice of $t_0$, this implies
equality.
\end{proof}

We now fix the interval $I$ from Lemma \ref{lemma1} and we set $Q = \#(t_0)$.
We may relabel the points $b_p(t)$ such that
\[
\lambda(t) = \sum_{p = 1}^Q \sigma_p(t) \delta_{b_p(t)}.
\]
The continuity of $\lambda(t)$ and the fact that $\sigma_p(t) \in 4\pi\Z \backslash \{0\}$
then imply that $b_p(t)$ is continuous in $I$ (if labeled appropriately)
for $p = 1,\ldots,Q$.

\begin{lemma} \label{lemma3}
For almost all $t \in I$ and all $\phi \in C_0^\infty(\Omega_r(a(t_1)))$,
\[
\sum_{p = Q + 1}^P \nabla^\perp \nabla \phi(b_p(t)) : B_p(t) = 0.
\]
\end{lemma}

\begin{proof}
It suffices to prove this for a $\phi$ with $b_p(t) \not\in \supp \phi$ for $p = 1,\ldots,Q$.
But in this case, the function
\[
t \mapsto \sum_{p = 1}^Q \sigma_p(t) \phi(b_p(t))
\]
is constant in a neighborhood of $t$. The claim then follow from \eqref{*}.
\end{proof}

\begin{lemma} \label{lemma4}
In $(t_1,t_2)$,
\[
\frac{d}{dt} \lambda(t) = 0.
\]
\end{lemma}

\begin{proof}
If we insert functions $\phi$ in \eqref{*} with $\nabla^2 \phi(b_p(t)) = 0$ for $p = 1,\ldots,Q$
at a time $t \in I$, then by density arguments we obtain the required identity in $I$. But then the curves
$b_1(t),\ldots,b_Q(t)$ are constant in $I$, and we can extend the arguments beyond the end points
of $I$ (unless they coincide with $t_1$ or $t_2$).
\end{proof}

The remaining arguments are similar to \cite{Kurzke-Melcher-Moser-Spirn:11}.
Now consider \eqref{motion_law_positive_epsilon} again and insert $\phi,\psi \in C_0^\infty(\Omega)$ with $\nabla^\perp \phi = \nabla \psi$ and $\nabla^2 \psi = 0$ in $\Omega \backslash \Omega_r(a(t_1))$.
Then we obtain
\begin{multline*}
\pi \sum_{n = 1}^N \left(\alpha_0 \psi(a_n(t_2)) + 4q_n(t_2) \phi(a_n(t_2)) - \alpha_0 \psi(a_n(t_1)) - 4q_n(t_1) \phi(a_n(t_1))\right) \\
= \int_{t_1}^{t_2} \left(\int_\Omega \nabla^\perp \nabla \phi : d\mu_*(t) + \sum_{p = 1}^P \nabla^\perp \nabla \phi(b_p(t)) : B_p(t)\right) \, dt
\end{multline*}
for certain functions $q_1,\ldots,q_N : [t_1,t_2] \to \frac{1}{2} + \Z$, using Lemma 3 in \cite{Kurzke-Melcher-Moser-Spirn:11}
again.

If $\phi = \psi = 0$ in in $\Omega \backslash \Omega_r(a(t_1))$, then we have the identity
\[
\int_{t_1}^{t_2} \int_\Omega \nabla^\perp \nabla \phi : d\mu_*(t) \, dt = 0,
\]
again. It follows
that
\[
\sum_{p = 1}^P \int_{t_1}^{t_2} \nabla^\perp \nabla \phi(b_p(t)) : B_p(t) \, dt = 0
\]
for all $\phi,\psi$ with this property. But then the last identity must be
true for all $\phi,\psi \in C_0^\infty(\Omega)$. Hence we have in fact
\begin{multline*}
\pi \sum_{n = 1}^N \left(\alpha_0 \psi(a_n(t_2)) + 4q_n(t_2) \phi(a_n(t_2)) - \alpha_0 \psi(a_n(t_1)) - 4q_n(t_1) \phi(a_n(t_1))\right) \\
= \int_{t_1}^{t_2} \int_\Omega \nabla^\perp \nabla \phi : d\mu_*(t) \, dt.
\end{multline*}
for all $\phi, \psi \in C_0^\infty(\Omega)$ with $\nabla^\perp \phi = \nabla \psi$
in a neighborhood of the vortices. From this we see that each $q_n$ is locally constant away from $\Sigma$. Using \eqref{gradientRenormalized},
we derive the motion law
\[
(\alpha_0 + 4q_n i)\pi \dot{a}_n = - \dd{}{a_n} W(a,d), \quad n = 1,\ldots,N.
\]
Thus we have proved Theorem \ref{vortexmotionlawLLG}.

\section{Comments on the complex valued case}  \label{mixedglsection}

As remarked in the introduction, a similar analysis can be performed on the complex Ginzburg-Landau equation of mixed-type dynamics,
\begin{equation} \label{mixedgl2}
(\alpha_\e + i) \dd{u_\e}{t} = \Delta u_\e + {1 \over \e^2} u_\e \LC 1 - |u_\e|^2 \RC
\end{equation}
on $\Omega \subset \R^2$ with $\alpha_\e \log \frac{1}{\epsilon} \to \alpha_0$. The dynamics of vortices were established rigorously by Miot \cite{Miot:08} on $\R^2$ and by the authors \cite{Kurzke-Melcher-Moser-Spirn:08} in bounded domains.  In both papers the authors make use 
of compactness results for the Ginzburg-Landau energy $e_{\mathrm{gl}}(u_\e)$ and the Jacobian $J(u_\e)$, the associated conservation laws 
\begin{align}
\dd{}{t} \LC \frac 12 \LC |u|^2 - 1\RC \RC & = \operatorname{div} j(u_\e) -  \alpha_\e  \LC i u , \dd{u}{t} \RC \label{mass_conservation} \\
\dd{}{t} J(u) & = \curl \operatorname{div} \LC \nabla u \otimes \nabla u \RC + \alpha_\e \LC \dd{u}{t} , \nabla u \RC \label{Jacobian_conservation}\\
\dd{}{t} e_\e(u) & = - \alpha_\e \LV \dd{u}{t} \RV^2 + \operatorname{div} \LC \nabla u , \dd{u}{t} \RC,  \label{energy_conservation}
\end{align}
and well-preparedness of the initial data to pin vortices to the ODE.  

In order to prove Theorem~\ref{vortexmotionlawmixedgl} we essentially follow the arguments from the proof of Theorem~\ref{vortexmotionlawLLG}; however, there is one important improvement. Solutions of equation \eqref{mixedgl2} remain smooth for all times, and thus there is no bubbling and no
discontinuity in the coefficients of the ODE. Furthermore, due to the compactness results of Jerrard-Soner \cite{Jerrard-Soner:02} and continuity of the Jacobian, one finds that $J(u_\e) \to J(u_*)$. In the LLG case, on the other hand, we cannot prove $\omega(\m_\e) \to \omega(\m_*)$, as
there may be additional Dirac masses induced by bubbling. Consequently, equation \eqref{mixedgl2} permits a few shortcuts in the proof. Otherwise,
the arguments remain the same and we do not repeat them.

The procedure follows along the lines of the proof of Theorem~\ref{vortexmotionlawLLG}.  First we establish strong convergence of the stress energy tensor ${j(u_\e) \over |u_\e|} \otimes  {j(u_\e) \over |u_\e|}$ to $j(u_*) \otimes j(u_*)$ outside the vortex cores and a finite number of points.
The proof is essentially the same as for Lemma \ref{lemma:strong_convergence} and uses identity \eqref{mass_conservation}.
Combining the differential identities \eqref{Jacobian_conservation} and \eqref{energy_conservation}, in a similar way as in the proof of Theorem~\ref{vortexmotionlawLLG}, we show that the defect measure does not affect the motion of the vortices, and the motion law follows.

%
%
%
%
%
%
%
%

\end{document}